\documentclass{amsart}

\RequirePackage[OT1]{fontenc}
\RequirePackage{amsthm,amsmath}
\RequirePackage[numbers]{natbib}
\RequirePackage[colorlinks,citecolor=blue,urlcolor=blue]{hyperref}
\usepackage{graphicx}
\usepackage{latexsym,color}
\usepackage{graphicx}
\usepackage{amssymb}
 \usepackage{amsfonts}
 \usepackage{amsmath}
\usepackage{amsthm}
\newtheorem{theorem}{Theorem}[section]
 
 \newtheorem{lemma}[theorem]{Lemma}
 \newtheorem{remark}[theorem]{Remark}
\newtheorem{corollary}[theorem]{Corollary}
\newtheorem{proposition}[theorem]{Proposition}

\numberwithin{equation}{section}

\begin{document}

\title{A Note on Costs Minimization with Stochastic Target Constraints}
 \thanks{}
\author{  Yan Dolinsky,
Benjamin Gottesman and
    Ori Gurel-Gurevich \\
 Hebrew University
}%
\address{
 Department of Statistics, Hebrew University of Jerusalem\\
 {e.mail: yan.yolinsky@mail.huji.ac.il}
}
\address{
 Department of Statistics, Hebrew University of Jerusalem\\
 {e.mail: beni.gottesman@gmail.com}
}
\address{
 Department of Mathematics, Hebrew University of Jerusalem\\
 {e.mail: Ori.Gurel-Gurevich@mail.huji.ac.il}
}

\date{\today}
\begin{abstract}
We study the minimization of the expected costs under stochastic constraint at the terminal time.
The first and the main result says that for a power type of costs, the value function is the minimal positive solution of a second order
semi--linear ordinary differential equation (ODE). Moreover, we establish the optimal control.
In the second example we show that the case of exponential costs leads to a trivial optimal control.
\end{abstract}

\subjclass[2010]{49J15, 60H30, 93E20}
 \keywords{Optimal stochastic control, backward stochastic
differential equations }%

\maketitle
\markboth{Y.Dolinsky, B.Gottesman and  O.Gurel-Gurevich}{Costs Minimization with Stochastic Target Constraints}
\renewcommand{\theequation}{\arabic{section}.\arabic{equation}}
\pagenumbering{arabic}

\section{Introduction and Main Results}\label{sec:1}
This note was inspired by a series of papers
 which dealt with stochastic tracking problems; see, e.g.,
\cite{AFKP,AJK,AK,B1,B2,GHQ,GHS,KP,S} and the references therein.

Consider a complete probability space
$(\Omega,\mathcal F,\mathbb P)$
together with a standard one--dimensional Brownian motion
$W_t, t\geq 0$
and the Brownian filtration $\mathcal F^W_t:=\sigma\{W_u:u\leq t\}$
completed by the null sets.

For any $(T,x)\in (0,\infty)\times\mathbb R$ and a progressively measurable processes
$u=\{u_t\}_{t=0}^T$ which satisfies the integrability
condition $\int_{0}^T |u_t| dt<\infty$ a.s. we
denote
$$X^{x,u}_t:=x+\int_{0}^t u_s ds, \ \ t\in [0,T].$$

For any $(T,x,c)\in (0,\infty)\times\mathbb R^2$ let $U(T,x,c)$ be the set of all progressively measurable processes
$u=\{u_t\}_{t=0}^T$ (with the above integrability condition) which satisfy
$X^{x,u}_T\geq\mathbb I_{W_T>c}$ a.s. As usual,
we set $\mathbb I_Q=1$ if an event $Q$ occurs and $\mathbb I_Q=0$ if not.

For a given $p>1$ introduce the stochastic control problem
\begin{equation}\label{control}
v(T,x,c):=\inf_{u\in U(T,x,c)} \mathbb E \left[\int_{0}^T |u_t|^p dt\right].
\end{equation}

For a given $(T,x,c)\in (0,\infty)\times [0,1]\times\mathbb R$ we say that $u\in U(T,x,c)$ is optimal if
$\mathbb E[\int_{0}^ T |u_t|^p dt]=v(T,x,c).$
Let $U^{+}(T,x,c)\subset U(T,x,c)$ be the set of all
$u\in U(T,x,c)$ such that $u\geq 0$ $dt\otimes \mathbb P$ a.s. and
$X^{x,u}_T\leq 1$ a.s.
\begin{lemma}\label{lem.optimal}
For any $(T,x,c)\in (0,\infty)\times [0,1]\times\mathbb R$
$$v(T,x,c)=\inf_{u\in U^{+}(T,x,c)} \mathbb E \left[\int_{0}^T |u_t|^p dt\right].$$
\end{lemma}
\begin{proof}
Let $u\in U(T,x,c)$.
Define
$$\hat u_t:=\max(0,u_t), \ \ t\in [0,T],$$
$$\theta:=T\wedge\inf\{t: X^{x,\hat u}_T=1\},$$
and
$$\tilde u_t:= \hat u_t\mathbb I_{t<\theta}, \ \ t\in [0,T].$$
Observe that,
$$X^{x,\tilde u}_T\geq 1\wedge X^{x,u}_T\geq \mathbb{I}_{W_T>c}$$
and
$$\mathbb E\left[\int _{0}^T  |\tilde u_t|^pdt\right]\leq \mathbb E\left[\int _{0}^T  |u_t|^p dt\right].$$ This completes the proof.
\end{proof}
The following Proposition will be crucial for deriving the main results.
\begin{proposition}\label{prop2.1}
For any $(T,x,c)\in (0,\infty)\times [0,1]\times\mathbb R$
$$v(T,x,c)=\frac{(1-x)^p}{T^{p-1}} v\left(1,0,\frac{c}{\sqrt T}\right).$$
\end{proposition}
\begin{proof}
The statement is obvious for $x=1$. Thus assume that $x<1$.

We use the scaling property of Brownian motion.
Define the Brownian motion
$B_t:=\frac{W_{t T}}{\sqrt T}$, $t\geq 0$. Let
$\mathcal F^B_t:=\sigma\{B_u:u\leq t\}$
be the filtration generated by $B$ completed with the null sets. Clearly,
$\mathcal F^B_{t}=\mathcal F^W_{t T}$, $t\geq 0$.
Let $\tilde U$
be the set of all stochastic processes $\tilde u=\{\tilde u_t\}_{t=0}^1$ which are
non negative, progressively measurable with respect to $\mathcal F^B$ and satisfy
$$\mathbb I_{B_1>\frac{c}{\sqrt T}}\leq \int_{0}^1 \tilde u_t dt \leq 1.$$
We notice that there is a bijection $U^{+}(T,x,c)\leftrightarrow\tilde U$
 which is given by
 $$ u_t=\frac{(1-x)\tilde u_{\frac{t}{T}}}{T}, \  t\in [0,T].$$
Thus, from Lemma \ref{lem.optimal}
\begin{eqnarray*}
&v(T,x,c)=\min_{ u\in U^{+}(T,x,c)} \mathbb E\left[\int_{0}^T u^p_t dt\right]\\
&=\min_{\tilde u\in \tilde U} \frac{(1-x)^p}{T^{p-1}}\mathbb E\left[\int_{0}^1 \tilde u^p_t dt\right]=
\frac{(1-x)^p}{T^{p-1}}v\left(1,0,\frac{c}{\sqrt T}\right).
\end{eqnarray*}
\end{proof}
Next, let $\Phi(\cdot)=\frac{1}{\sqrt 2\pi}\int_{-\infty}^{\cdot} e^{-\frac{y^2}{2}} dy$ be the
 cumulative distribution function of the standard normal distribution.
For any $T>0$ and $c\in\mathbb R$ consider the martingale $\{M^{T,c}_t\}_{t=0}^T$ given by
\begin{equation}\label{martingale}
M^{T,c}_T=\mathbb I_{W_T<c}, \ \ M^{T,c}_t=\mathbb P(W_T<c|\mathcal F^W_t)=\Phi\left(\frac{c-W_t}{\sqrt{T-t}}\right), \ t\in [0,T).
\end{equation}

Define the function $g:(0,1)\rightarrow\mathbb R_{+}$ by
 $$g(z)=v\left(1,0,\Phi^{-1}(z)\right)$$ where $\Phi^{-1}$ is the inverse function.
From Proposition \ref{prop2.1} we have
\begin{equation}\label{main}
 v(t,x,c)= \frac{(1-x)^p}{t^{p-1}} g\left(\Phi\left(\frac{c}{\sqrt t}\right)\right)=
 \frac{(1-x)^p}{t^{p-1}} g(M^{t,c}_0), \ \ \forall (t,x,c)\in (0,\infty)\times [0,1]\times\mathbb R.
 \end{equation}

Now, we are ready to state the main results which will be proved in Section \ref{sec:4}.
 \begin{theorem}\label{thm2.1}
 ${}$\\
 (I) Let $h:(0,1)\rightarrow \mathbb R_{+}$ be given by
 $h(y)=\frac{\exp\left(-[\Phi^{-1}(y)]^2\right)}{4\pi}.$
The function $g:(0,1)\rightarrow\mathbb R_{+}$ is a non increasing solution of the ODE
 \begin{equation}\label{ode}
 h(y)g''(y)+(p-1)\left(g(y)-g^{\frac{p}{p-1}}(y)\right)=0, \ \ y\in (0,1)
 \end{equation}
 with the boundary conditions
 \begin{equation}\label{boundary}
 \lim_{y\rightarrow 0} g(y)=1 \ \mbox{and} \ \lim_{y\rightarrow 1}g(y)=0.
 \end{equation}
 Moreover, the following minimality holds. If $\hat g:(0,1)\rightarrow\mathbb R_{+}$ is another solution
 to (\ref{ode}) ($\hat g(y_0)\neq g(y_0)$ for some $y_0$) and satisfies
 \begin{equation}\label{minimal}
\lim_{y\rightarrow 0} \hat g(y)>0
\end{equation}
 then $\hat g(y)>g(y)$ for all $y\in (0,1)$.\\
(II)
Let $(T,x,c)\in (0,\infty)\times [0,1]\times\mathbb R$.
The optimal control is given by
$$\hat u_t=(1-x)\frac{g^{\frac{1}{p-1}}(M^{T,c}_t)}{T-t}\exp\left(-\int_{0}^t \frac{g^{\frac{1}{p-1}}(M^{T,c}_s)}{T-s}ds\right), \ \ \ t\in [0,T).$$
Namely, for the optimal control we have the ODE:
$$\frac{d X^{x,\hat u}_t}{dt}=g^{\frac{1}{p-1}}(M^{T,c}_t)\frac{1-X^{x,\hat u}_t}{T-t}, \  \ \ t\in [0,T). $$
(III) Let $T>0$ and $c\in\mathbb R$. Then the pair
$(Y,Z)$ given by
$$Y_t:= \frac{g(M^{T,c}_t)}{(T-t)^{p-1}}, \ \ Z_t:=-\frac{g'(M^{T,c}_t)e^{-\frac{(c-W_t)^2}{2(T-t)}}}{\sqrt{2\pi(T-t)^{2p-1}}}, \ \  t\in [0,T)$$
 is the minimal solution of the backward stochastic
differential equation (BSDE)
\begin{equation}\label{BSDE}
dY_t=(p-1)Y^{\frac{p}{p-1}}_t dt+Z_t dW_t, \ \ t\in [0,T)
\end{equation}
with the singular terminal condition $Y_T=\infty \mathbb I_{W_T>c}$. This terminal condition means that $\lim_{t\rightarrow T}Y_t=\infty \mathbb I_{W_T>c}$ a.s.
where we use the convention $\infty \cdot 0:=0$.
\end{theorem}
\begin{remark}\label{uni}
It is easy to see that the optimal control is unique. Indeed, if by contradiction $u,\tilde u\in U(T,x,c)$ are optimal controls
and $dt\otimes \mathbb P(u\neq \tilde u)>0$. Then, the process $\frac{u+\tilde u}{2}$ satisfies $\frac{u+\tilde u}{2}\in U(T,x,c)$ and
from the strict convexity of the function $z\rightarrow |z|^p$ we have
$$\mathbb E\left[\int_{0}^T\left|\frac{u_t+\tilde u_t}{2}\right|^ pdt\right]<
\frac{1}{2}\left(\mathbb E\left[\int_{0}^ T |u_t|^p dt\right]+\mathbb E\left[\int_{0}^ T |\tilde u_t|^p dt\right]\right)=v(T,x,c)
$$
which is a contradiction.
\end{remark}
\begin{remark}
A natural question is whether there exists a unique positive, non increasing solution to the
ODE (\ref{ode}) with the boundary conditions
(\ref{boundary}).
Due to the fact that $h$ takes the value $0$ at the end points $\{0,1\}$ the uniqueness seems to be far from obvious
and we leave it for future research.
\end{remark}

\section{Proof of the Main Results}\label{sec:4}
We start with the following regularity result.
\begin{lemma}\label{lem.convex}
The function $g:(0,1)\rightarrow \mathbb R_{+}$ is concave, non increasing and satisfies
 $\lim_{y\rightarrow 0} g(y)=1$.
  \end{lemma}
\begin{proof}
The fact that $g$ is non increasing is obvious.

Next, we establish the equality $\lim_{y\rightarrow 0} g(y)=1$.
From the Jensen inequality it follows that for any $u\in U(1,0,c)$
$$\mathbb E\left[\int_{0}^1 |u_t|^p dt\right]\geq (\mathbb P(W_1>c))^p=(1-\Phi(c))^p.$$ Thus,
$g(y)=v(1,0,\Phi^{-1}(y))\geq (1-y)^p$ and we conclude that $\lim_{y\rightarrow 0} g(y)=1$.

It remains to prove concavity.
Fix $a_1<a<a_2$. Let us show that
\begin{equation*}
g(a)\geq g(a_1)\frac{a_2-a}{a_2-a_1}+g(a_2)\frac{a-a_1}{a_2-a_1}.
\end{equation*}
Let $c=\Phi^{-1}(a)$. Choose $\epsilon>0$. There exists
$u\in U^{+}(1,0,c)$ such that
\begin{equation}\label{3.2}
g(a)>\mathbb E \left[\int_{0}^1 |u_t|^p dt\right]-\epsilon.
\end{equation}

Consider the martingale $M:=M^{1,c}$  given by (\ref{martingale}). Observe that $M_0=a$. Define the stopping time
$$\tau=\inf\{t: M_t\notin(a_1,a_2)\}.$$
Clearly, $\tau<1$ a.s.
and so from the equality $\mathbb E[M_{\tau}]=M_0$ we conclude that
\begin{equation}\label{3.3-}
\mathbb P(M_{\tau}=a_1)=\frac{a_2-a}{a_2-a_1} \  \mbox{and} \ P(M_{\tau}=a_2)=\frac{a-a_1}{a_2-a_1}.
\end{equation}
Next, let
$D=\int_{0}^{\tau} u_t dt$.
From the Holder inequality
\begin{equation}\label{3.-}
\int_{0}^\tau |u_t|^p dt\geq \frac{D^p}{\tau^{p-1}} \ \ \mbox{a.s.}
\end{equation}
From (\ref{main}),
the fact that $\{W_{s+\tau}-W_{\tau}\}_{s=0}^\infty$ is a Brownian motion independent of $\mathcal F^W_{\tau}$, and the inequality
$D+\int_{\tau}^{1}u_t dt\geq \mathbb I_{W_1-W_{\tau}>c-W_{\tau}}$ (notice that $D\in [0,1]$) we get
\begin{eqnarray*}
&\mathbb E \left[\int_{\tau}^1 |u_t|^p dt\big|\mathcal F^W_{\tau}\right]\geq v\left(1-\tau,D,c-W_{\tau}\right)\\
&=\frac{\left(1-D\right)^{p}}{(1-\tau)^{p-1}} g\left(\Phi\left(\frac{c-W_{\tau}}{\sqrt{1-\tau}}\right)\right)=
\frac{\left(1-D\right)^{p}}{(1-\tau)^{p-1}} g(M_{\tau}).
\end{eqnarray*}
Thus,
\begin{equation}\label{3.3}
\mathbb E \left[\int_{\tau}^1 |u_t|^p dt\right]\geq \mathbb E\left[\frac{\left(1-D\right)^{p}}{(1-\tau)^{p-1}} g(M_{\tau})\right].
\end{equation}
By combining (\ref{3.2})--(\ref{3.3}), the fact that $g\leq 1$ and the simple inequality
$\frac{z^p}{y^{p-1}}+\frac{(1-z)^p}{(1-y)^{p-1}}\geq 1$ for $0<y,z<1$ we obtain
 \begin{eqnarray*}
&g(a)+\epsilon>\mathbb E\left[\int_{0}^1 |u_t|^p dt\right]\geq \mathbb E
\left[\left(\frac{D^p}{\tau^{p-1}}+\frac{\left(1-D\right)^{p}}{(1-\tau)^{p-1}} \right) g(M_{\tau})\right]\\
&\geq\mathbb E[g(M_{\tau})]=g(a_1)\frac{a_2-a}{a_2-a_1}+g(a_2)\frac{a-a_1}{a_2-a_1}.
\end{eqnarray*}
Since $\epsilon>0$ was arbitrary we complete the proof.
\end{proof}

The proof of the main results will be based on the theory developed in \cite{KP}.
We start with preparations.
For any $(T,x,c)\in (0,\infty)\times [0,1]\times\mathbb R$
introduce the optimal position targeting problem
\begin{equation*}
\hat v(T,x,c):=\inf_{u}\mathbb E \left[\int_{0}^T | u_t|^p dt+\xi |X^{x,u}_T|^p\right]
\end{equation*}
where the infimum is taken over all progressively measurable processes $u=\{u_t\}_{t=0}^T$,
$\xi:=\infty\mathbb I_{W_T>c}$ and as before, we use the convention
$\infty\cdot 0:=0$.

Using same arguments as in Lemma \ref{lem.optimal} gives that
$$\hat v(T,x,c)= \inf_{u\in \hat U^{-}(T,x,c)}\mathbb E \left[\int_{0}^T | u_t|^p dt+\xi |X^{x,u}_T|^p\right]$$
where
$\hat U^{-}(T,x,c)$ is the set of all progressively measurable processes
$u=\{u_t\}_{t=0}^T$ such that $u\leq 0$ $dt\otimes \mathbb P$ a.s.,
$X^{x,u}_T\geq 0$ a.s. and $X^{x,u}_T=0$ on the event $\{W_T>c\}$.

Clearly, there is a bijection
$U^{+}(T,x,c)\leftrightarrow \hat U^{-}(T,1-x,c)$ given by $u\leftrightarrow -u$. Moreover, for any $u\in U^{-}(T,1-x,c)$ we have
\begin{equation}\label{lin}
\mathbb E \left[\int_{0}^T |-u_t|^p dt\right]=\mathbb E \left[\int_{0}^T |u_t|^p dt+\xi |X^{1-x,u}_T|^p\right].
\end{equation}
Thus, from Lemma \ref{lem.optimal} we conclude that
\begin{equation}\label{link}
v(t,x,c)=\hat v(t,1-x,c), \ \ \ \forall (t,x,c)\in (0,\infty)\times [0,1]\times\mathbb R.
\end{equation}

This brings us to the following corollary.
\begin{corollary}\label{co}
Let $T>0$ and $c\in\mathbb R$.
There exists a progressively measurable process $\{Z_t\}_{0\leq t<T}$ such that
the pair $\left(\frac{g(M^{T,c}_t)}{(T-t)^{p-1}},Z_t\right)_{0\leq t<T}$
is the minimal supersolution (see Definition 1 in \cite{KP}) to the BSDE given by (\ref{BSDE}) with the singular terminal condition
$Y_T=\infty\mathbb I_{W_T>c}$.
\end{corollary}
\begin{proof}
 From Theorem 3 in \cite{KP} it follows that there exists a minimal supersoution $(Y,Z)$ to the above BSDE. Moreover, by combining Theorem
 3 in \cite{KP} together with
the Markov property of Brownian motion,
(\ref{main}) and (\ref{link}) we obtain that
$Y_t= \frac{g(M^{T,c}_t)}{(T-t)^{p-1}}$, $t\in [0,T)$.
 \end{proof}
\begin{remark}
A priori we do not know that $g$ is continuously differentiable and so we can not apply the Ito formula and find $Z$.
In the proof of Theorem \ref{thm2.1} we will show that $g$ satisfies the ODE (\ref{ode}) and then we will find $Z$.
\end{remark}
Now, we are ready to prove Theorem \ref{thm2.1}.
\begin{proof}
Proof of Theorem \ref{thm2.1}.\\
${}$\\
\textbf{First step: Proving that the minimal supersolution is a solution.}\\
Fix $T>0$ and $c\in\mathbb R$. Let $\xi:=\infty \mathbb I_{W_T>c}$.
Let us show that the supersolution $(Y,Z)$ from Corollary \ref{co} is actually a solution. To that end,
we need to establish the
inequality
$\lim\sup_{t\rightarrow T}Y_t\leq\xi$.

We wish to apply Theorem 4 in \cite{P}. There is a technical problem that the indicator function is not continuous
and so condition (4) in \cite{P} does not hold. Still, this issue can be simply solved by the following density argument.
Define a sequence of functions $\phi^{(n)}:\mathbb R\rightarrow\mathbb R\cup\{\infty\}$, $n\in\mathbb N$ by
$$\phi^{(n)}(z)=\left(\frac{1}{c-z}-n\right)\mathbb  I_{c-\frac{1}{n}\leq z<c}+\infty \mathbb  I_{z\geq c}.$$
Observe that for any $n$, $\phi^{(n)}$ satisfies condition (4) in \cite{P}. Hence,
from Theorem 4 in \cite{P} there exists a pair $(Y^{(n)},Z^{(n)})$ which satisfies the BSDE (\ref{BSDE})
with the terminal constraint $Y^{(n)}_T=\phi^{(n)}(W_T)$.
Since $\phi^{(n)}(W_T)\geq \xi $, then from the minimality property of $(Y,Z)$ we obtain that for any $n$,
$Y_t\leq Y^{(n)}_t$ a.s. for any $t\in [0,T)$.
Thus,
$$\lim\sup_{t\rightarrow T}Y_t\leq
\lim\inf_{n\rightarrow\infty} Y^{(n)}_T=\lim_{n\rightarrow\infty} \phi^{(n)}(W_T)=\xi, $$
as required.  \\
${}$\\
\textbf{Second step: Establishing statement (I) in Theorem \ref{thm2.1}.}\\
From Corollary \ref{co} and the previous step it follows that
$\lim_{t\rightarrow 1}\frac{g(M^{1,0}_t)}{(1-t)^{p-1}}=0$ on the event
$\{W_1<0\}$. Clearly,
 $\lim_{t\rightarrow  1} M^{1,0}_t=1$ on the event $\{W_1<0\}$, and so we conclude that
 $\lim_{y\rightarrow 1}g(y)=0.$ This together with the boundary condition
 $\lim_{y\rightarrow 0}g(y)=1$ (was established in Lemma \ref{lem.convex})
 gives (\ref{boundary}).

 Next, we prove (\ref{ode}).
Extend the function $g$ to the closed interval $[0,1]$ by
$g(0):=1$
and $g(1):=0$.
Choose $a\in (0,1)$. Let $c=\Phi^{-1}\left(\frac{a}{2}\right)$.
Consider the martingale $M:=M^{1,c}$.
From Lemma \ref{lem.convex} it follows that
$g:[0,1]\rightarrow [0,1]$ is concave and continuous. Thus, $g(M_t)$, $t\in [0,1]$ is a continuous and uniformly integrable
super--martingale. From Doob's decomposition
$$g(M_t)=N_t-A_t, \ \ t\in [0,1]$$
where $N=\{N_t\}_{t=0}^1$ is a martingale and $A=\{A_t\}_{t=0}^1$ is a continuous,
non decreasing process with
$A_0=0$.

Recall the minimal supersolution $\left(\frac{g(M_t)}{(1-t)^{p-1}},Z_t\right)_{0\leq t<1}$
from Corollary \ref{co}.
From the product rule and (\ref{BSDE}) we get
\begin{eqnarray*}
&(p-1)\frac{g^{\frac{p}{p-1}}(M_t)}{(1-t)^p} dt+Z_t dW_t=d\left(\frac{g(M_t)}{(1-t)^{p-1}}\right)\\
&=\frac{dN_t}{(1-t)^{p-1}}-\frac{dA_t}{(1-t)^{p-1}}+(p-1)\frac{g(M_t)}{(1-t)^{p}}dt.
\end{eqnarray*}
Hence,
$$\frac{dA_t}{dt}=(p-1)\frac{
g(M_t)-g^{\frac{p}{p-1}}(M_t)}{1-t}.$$
We conclude that,
\begin{equation}\label{4.new}
g(M_t)=N_t-(p-1)\int_{0}^t\frac{g(M_s)-g^{\frac{p}{p-1}}(M_s)}{1-s}ds \ \ \forall t\in [0,1] \ \ \mathbb P \ \mbox{a.s.}
\end{equation}
Next, observe that
$M_0=\frac{a}{2}$ and define the function
$f:\left[\frac{a}{3},\frac{1+a}{2}\right]\rightarrow\mathbb R$ by
$$f(y)=-(p-1)\int_{\beta=a/3}^y \int_{\alpha=\frac{a}{3}}^\beta\frac{g(\alpha)-g^{\frac{p}{p-1}}(\alpha)}{h(\alpha)}d\alpha d\beta.$$
Notice that  $f\in C^2 \left[\frac{a}{3},\frac{1+a}{2}\right]$ and
$f''(y)=-(p-1)\frac{g(y)-g^{\frac{p}{p-1}}(y)}{h(y)}$, $y\in \left[\frac{a}{3},\frac{a+1}{2}\right]$.

For any $y\in \left(M_0,\frac{1+a}{2}\right)$ consider the stopping time
$\tau_y=\inf\{t: M_t\notin\left(\frac{a}{3},y\right)\}$. Clearly, $\tau_y<T$ a.s.

We notice that $\frac{d\langle M\rangle}{dt}=\frac{2 h(M_t)}{1-t}$, and so from the Ito Formula and (\ref{4.new}) we obtain
$$g(M_{\tau_y})-f(M_{\tau_y})=N_{\tau_y}-f(M_0)-\int_{0}^{\tau_y} f'(M_t) dM_t.$$
Hence,
\begin{equation}\label{4.92}
\mathbb E[g(M_{\tau_y})]-\mathbb E[f(M_{\tau_y})]=g(M_0)-f(M_0).
\end{equation}
Similarly, to (\ref{3.3-})
$$\mathbb P(M_{\tau_y}=y)=\frac{M_0-\frac{a}{3}}{y-\frac{a}{3}}
\  \mbox{and} \
\mathbb P\left(M_{\tau_y}=\frac{a}{3}\right)=\frac{y-M_0}{y-\frac{a}{3}} .$$
This together with (\ref{4.92}) yields that $g(y)-f(y)$ is a linear function on the interval $\left(M_0,\frac{1+a}{2}\right)$. In particular
$$g''(a)=f''(a)=-(p-1)\frac{g(a)-g^{\frac{p}{p-1}}(a)}{h(a)}.$$ Since $a\in (0,1)$ was arbitrary
we complete the proof of (\ref{ode}). \\

Finally, we prove minimality.
Assume that there exists a positive function $\hat g\neq g$ which satisfies
(\ref{ode}) and
(\ref{minimal}).
Define
the pair $(\hat Y,\hat Z)$ by
$$\hat Y_t:= \frac{\hat g(M^{1,0}_t)}{(1-t)^{p-1}}, \ \ \hat Z_t:=-\frac{\hat g'(M^{1,0}_t)e^{-\frac{W^2_t}{2(1-t)}}}{\sqrt{2\pi(1-t)^{2p-1}}}, \ \  t\in [0,1).$$
From the Ito formula ($\hat g$ satisfies (\ref{ode}) and so continuously differentiable)
it follows that the pair $(\hat Y,\hat Z)$
is a supersolution to
the BSDE (\ref{BSDE}) with the singular terminal condition $\hat Y_T=\infty \mathbb I_{W_1>0}$. From Corollary \ref{co} we conclude that
$\frac{\hat g(M^{1,0}_t)}{(1-t)^{p-1}}\geq \frac{g(M^{1,0}_t)}{(1-t)^{p-1}}$ a.s. for any $t\in [0,1)$. Thus,
$g(y)\geq\hat g(y)$ for all $y\in (0,1)$.

Let us argue strict inequality.
Indeed, assume by contradiction that there is $y_0\in (0,1)$ for which $\hat g(y_0)=g(y_0)$, then clearly $y_0$ is a minimum
point for the function $\hat g-g$. Hence, $\hat g'(y_0)=g'(y_0)$. Since $h(y)$ is bounded away from zero if $y$ is bounded away from the end points $\{0,1\}$, then
from standard uniqueness for initial value problems
we conclude that
$\hat g=g$ on the interval $(0,1)$. This is a contradiction and the proof of (I) is completed.
${}$\\
${}$\\
\textbf{Third step: Completion of the proof.}
\\
In this step we complete the proof of statements (II)--(III) in Theorem \ref{thm2.1}.
Since
$g$ is continuously differentiable (satisfies (\ref{ode})) then from the Ito formula,
Corollary \ref{co} and the first step of the proof we obtain statement (III).

It remains to prove Statement (II).
Let $(T,x,c)\in (0,\infty)\times [0,1]\times\mathbb R$ and let $\xi:=\infty \mathbb I_{W_T>c}$.
From
Theorem 3 in \cite{KP} and Corollary \ref{co} it follows that the optimal control for the optimization problem
$$\hat v(T,1-x,c)= \inf_{u\in \hat U^{-}(T,1-x,c)}\mathbb E \left[\int_{0}^T | u_t|^p dt+\xi |X^{x,u}_T|^p\right]$$
is given by
$$u_t=\frac{dX^{1-x, u}_t}{dt}=-(1-x)\frac{g^{\frac{1}{p-1}}(M^{T,c}_t)}{T-t}\exp\left(-\int_{0}^t \frac{g^{\frac{1}{p-1}}(M^{T,c}_s)}{T-s}ds\right), \ \ t\in [0,T].$$
From (\ref{lin})--(\ref{link}) we obtain that
$$\hat u_t:=-u_t=(1-x)\frac{g^{\frac{1}{p-1}}(M^{T,c}_t)}{T-t}\exp\left(-\int_{0}^t \frac{g^{\frac{1}{p-1}}(M^{T,c}_s)}{T-s}ds\right), \ \ t\in [0,T)$$
is the optimal control for the optimization problem (\ref{control}), as required.
\end{proof}
\section{The Exponential Case}\label{sec:5}
Let $\lambda>0$
and consider
the optimization problem
\begin{equation*}
w(T,x,c):=\inf_{u\in U(T,x,c)}\mathbb E\left[\int_{0}^T \left(e^{\lambda |u_t|}-1\right)dt\right].
\end{equation*}
Namely,
we consider a stochastic target problem with exponential
costs $z\rightarrow e^{\lambda |z|}-1$ and the same stochastic target as in (\ref{control}).

The following result says that for any $(T,x,c)$
the optimal control is targeting towards $1$ with a constant speed.
\begin{theorem}
Let $(T,x,c)\in (0,\infty)\times\mathbb R^2$. Then
\begin{equation*}
w(T,x,c)=T\left(e^{\frac{\lambda(1-x)^{+}}{T}}-1\right),
\end{equation*}
and the unique optimal control is given by
$u= \frac{{(1-x)}^{+}}{T}$ $dt\otimes\mathbb P$ a.s.
\end{theorem}
\begin{proof}
Choose $(T,x,c)\times (0,\infty)\times \mathbb R^2$.
The statement is obvious for $x\geq 1$. Hence, without loss of generality we assume that $x<1$.
The cost function
is strictly convex, and so,
by using the same arguments as in Remark \ref{uni} we obtain that the optimal control is unique.
Thus, in order to prove the theorem it is
sufficient to show that the value function satisfies the inequality
\begin{equation}\label{5.1}
w(T,x,c)\geq T\left(e^{\frac{\lambda(1-x)}{T}}-1\right).
\end{equation}
Let $\mathcal C$ be the set of all adapted, continuous and uniformly bounded
processes.
Let $\mathcal M$ be the set of all strictly positive and uniformly bounded martingales
$\mathbb M=\{\mathbb M_t\}_{t=0}^T$ with $\mathbb M_0=1$.

Applying the standard technique of Lagrange multipliers we obtain
\begin{eqnarray*}
&w(T,x,c)\\
&\geq \inf_{C\in\mathcal C}\sup_{\alpha>0}\sup_{\mathbb M\in\mathcal M}
\mathbb E\left[\int_{0}^T \left(e^{\lambda |C_t|}-1\right) dt-\alpha\mathbb M_T\left(x+\int_{0}^T C_t dt-\mathbb I_{W_T>c}\right)\right]\\
&\geq \sup_{\alpha>0}\sup_{\mathbb M\in\mathcal M}\inf_{C\in\mathcal C} \mathbb E\left[\int_{0}^T \left(e^{\lambda |C_t|}-1\right) dt
-\alpha\mathbb M_T\left(x+\int_{0}^T C_t dt-\mathbb I_{W_T>c}\right)\right]\\
&\geq\sup_{\alpha>0}\sup_{\mathbb M\in\mathcal M}\inf_{C\in\mathcal C} \mathbb E\left[\int_{0}^T \left(e^{\lambda C_t}-1\right) dt-\alpha x-\alpha\int_{0}^T \mathbb M_t C_t dt+\alpha\mathbb{M}_T\mathbb I_{W_T>c}\right].\\
\end{eqnarray*}
Observe that for a given $\alpha>0$ and a martingale $\mathbb M$ the minimum of the above expression is obtained by taking
$C_t=\frac{\ln (\alpha\mathbb M_t/\lambda)}{\lambda}$, $t\in [0,T]$.
Hence,
\begin{eqnarray*}
&w(T,x,c)+T\\
&\geq\sup_{\alpha>0}\sup_{\mathbb M\in\mathcal M}\left[
\mathbb E\left(\alpha\mathbb{M}_T\mathbb I_{W_T>c}+\int_{0}^T\left(\frac{\alpha}{\lambda}\mathbb M_t-\frac{\alpha}{\lambda}\ln(\alpha/\lambda)\mathbb M_t-\frac{\alpha}{\lambda} \mathbb M_t\ln\mathbb M_t\right)dt\right)-\alpha x\right]\\
&=\sup_{\alpha>0}\sup_{\mathbb M\in\mathcal M}
\left[\alpha\mathbb E\left(\mathbb{M}_T\mathbb I_{W_T>c}-
\frac{1}{\lambda}\int_{0}^T \mathbb M_t\ln\mathbb M_tdt\right)+\frac{\alpha T}{\lambda}\left(1+\ln\lambda-\ln\alpha\right)-\alpha x\right].
\end{eqnarray*}
Clearly for a given $z_1\in\mathbb R$ and $z_2>0$ we have
$\max_{\alpha>0}[\alpha z_1-z_2\alpha\ln\alpha]=z_2e^{\frac{z_1}{z_2}-1}$.
We conclude that
$$w(T,x,c)+T\geq T e^{-\frac{\lambda x}{T}} \sup_{\mathbb M\in\mathcal M}\exp\left[\frac{1}{T}\mathbb E\left(\lambda\mathbb{M}_T\mathbb I_{W_T>c}-
\int_{0}^T \mathbb M_t\ln\mathbb M_tdt\right)\right]$$
and (\ref{5.1}) follows from the following lemma.
\end{proof}
\begin{lemma}
For any $\epsilon>0$ there exists $\mathbb M\in\mathcal M$ such that
\begin{equation}\label{5.2}
\mathbb E\left[\mathbb M_T\mathbb \mathbb I_{W_T>c}\right]>1-\epsilon
\end{equation}
and
\begin{equation}\label{5.3}
\mathbb E\left[\int_{0}^T
 \mathbb M_t\ln \mathbb M_t dt\right]<\epsilon.
 \end{equation}
 \end{lemma}
 \begin{proof}
Choose $\epsilon>0$. First, assume that we found a strictly positive martingale $\mathbb M$ with $\mathbb M_0=1$ which satisfy (\ref{5.2})--(\ref{5.3}).
 Then for any $N\in\mathbb N$ define $\mathbb M^{(N)}\in\mathcal M$ by
 $\mathbb M^{(N)}_t:=\mathbb M_{t\wedge\sigma_N}$, $t\in [0,T]$ where
 $\sigma_N:=T\wedge\inf\{t: \mathbb M_t=N\}$. Clearly,
 $$\mathbb M^{(N)}_t=\mathbb E\left[\mathbb M_t|\mathcal F^W_{\sigma_N}\right], \ \ t\in [0,T].$$
  Thus, from the Jensen inequality for the function $z\rightarrow z\ln z$ and the Fubini theorem
 $$
\mathbb E\left[\int_{0}^T
 \mathbb M^{(N)}_t\ln \mathbb M^{(N)}_t dt\right]\leq \mathbb E\left[\int_{0}^T
 \mathbb M_t\ln \mathbb M_t dt\right]<\epsilon.$$
 Next,
 from the Fatou Lemma and the fact that $\sigma_N\uparrow T$ as $n\rightarrow\infty$
 $$\mathbb E\left[\mathbb M_T\mathbb \mathbb I_{W_T>c}\right]\leq\lim\inf_{N\rightarrow \infty}
 \mathbb E\left[\mathbb M^{(N)}_T\mathbb \mathbb I_{W_T>c}\right]. $$
 We conclude that in order to prove the statement, it is sufficient to find a strictly positive martingale which satisfy (\ref{5.2})--(\ref{5.3}).

To this end, consider a strictly positive martingale of the form
 $$\mathbb M_t:=e^{\int_{0}^t \zeta_u dW_u-\int_{0}^t\frac{1}{2}\zeta^2_u du}, \ \ t\in[0,T]$$
 where $\{\zeta_t\}_{t=0}^1$ is a continuous deterministic function. There exists a probability measure $\mathbb Q$ such that
 $\frac{d\mathbb Q}{d\mathbb P}|\mathcal F_t=\mathbb M_t.$ Moreover, from the Girsanov theorem
 the process
 $\tilde W_t:=W_t-\int_{0}^t \zeta_u du$, $t\in [0,T]$ is a Brownian motion under $\mathbb Q$. Thus,
\begin{equation}\label{5.100}
 \mathbb E\left[\mathbb M_T\mathbb \mathbb I_{W_T>c}\right]=\mathbb Q\left(W_T>c\right)=\mathbb Q\left(\tilde W_T+\int_{0}^T \zeta_t dt>c\right)
\end{equation}
and
\begin{eqnarray}\label{5.101}
   &\mathbb E\left[\int_{0}^T
 \mathbb M_t\ln \mathbb M_t dt\right]=\mathbb E_{\mathbb Q}\left[\int_{0}^T
 \ln \mathbb M_t dt\right]\nonumber\\
 &=\mathbb E_{\mathbb Q}\left[\int_{0}^T\left(\int_{0}^t\zeta_ud\tilde W_u+\frac{1}{2}\int_{0}^t \zeta^2_u du\right)dt \right]\nonumber\\
 &=\frac{1}{2}\left[\int_{0}^T\int_{0}^t \zeta^2_u du dt\right]=\frac{1}{2}\int_{0}^T \zeta^2_t(T-t) dt.
 \end{eqnarray}
Observe that for the sequence of continuous functions $\zeta^{(n)}:[0,T]\rightarrow\mathbb R$, $n\in\mathbb N$ given by
$$\zeta^{(n)}_t:=\frac{n^{-\frac{2}{3}}}{\left(T+\frac{1}{n^n}-t\right)^{1-\frac{1}{n}}}, \ \ t\in [0,T]$$
we have
$$\lim_{n\rightarrow \infty} \int_{0}^T \zeta^{(n)}_t dt\geq \lim_{n\rightarrow \infty} n^{\frac{1}{3}}\left(T^{\frac{1}{n}}-\frac{1}{n}\right)=\infty$$
and
$$\lim_{n\rightarrow \infty} \int_{0}^T [\zeta^{(n)}_t]^2(T-t) dt\leq \lim_{n\rightarrow \infty} n^{-\frac{4}{3}}\int_{0}^T \frac{dt}{(T-t)^{1-\frac{1}{n}}}=0.$$
This together with (\ref{5.100})--(\ref{5.101}) yields that for sufficiently large $n$ the martingale given by
$\mathbb M_t:=e^{\int_{0}^t \zeta^{(n)}_u dW_u-\frac{1}{2}[\zeta^{(n)}_u]^2 du}$, $t\in[0,T]$
satisfies (\ref{5.2})--(\ref{5.3}).

\end{proof}

\section{Numerical Results}
In this section we focus on the case of quadratic costs (i.e. $p=2$) and provide numerical results for the
value function and
simulations for
the optimal control.

From (\ref{main}) we have
$$g\left(\frac{1}{2}\right)=\inf_{u\in U(1,0,0)} \mathbb E \left[\int_{0}^1 u^2_t dt\right].$$
By approximating the Brownian motion with scaled random walks we compute numerically the right hand side of the above equality. The result is
$g\left(\frac{1}{2}\right)=0.88$. Then, we apply the shooting method and look for the correct value of the
derivative $g'\left(\frac{1}{2}\right)$. Namely we look for a real number $\gamma$ such that
the unique ($h\neq 0$ in the interval $(0,1)$) solution of the initial value problem
$$h(y)g''(y)+g(y)-g^2(y)=0, \ \ g\left(\frac{1}{2}\right)=0.88 \ \ \mbox{and} \ \ g'\left(\frac{1}{2}\right)=\gamma$$
will satisfy the boundary conditions $g(0)=1$ and $g(1)=0$. We get (numerically) a unique value $\gamma=-0.21$.
The result is illustrated in Figure 1.
\begin{figure}
\centering
\caption{A plot of the function $g$ for the case $p=2$.}
\includegraphics[width=0.9\textwidth]{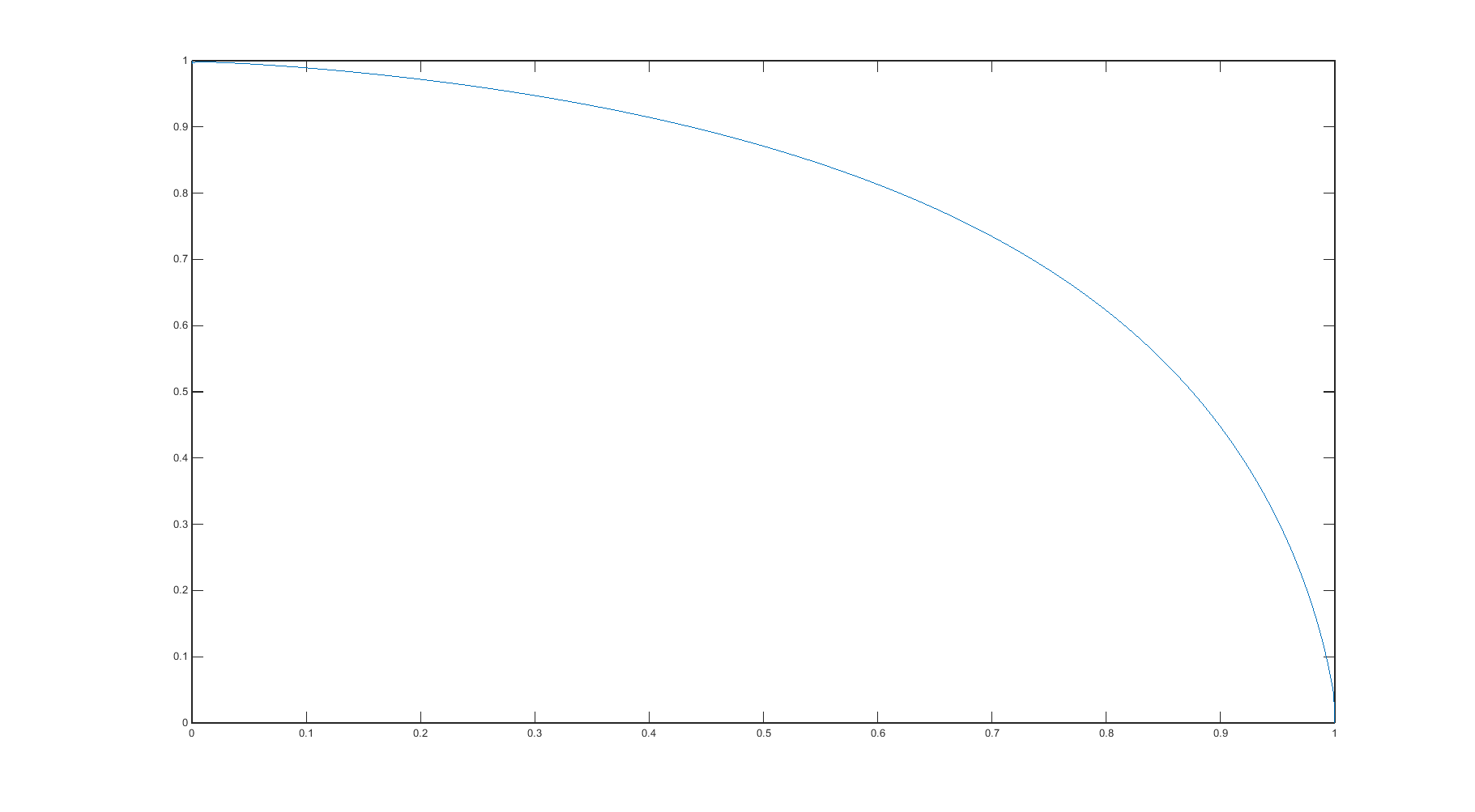}
\end{figure}

Next, for $T=1$ and $x=c=0$
we simulate a path of the optimal control $u\in U(1,0,0)$ and the corresponding strategy
$X^{0,u}_t=\int_{0}^t u_s ds$, $t\in [0,1]$.
This is done by simulating a Brownian path and applying Theorem \ref{thm2.1} (see Figures 2-3 below).
\begin{figure}
\centering
\includegraphics[width=0.9\textwidth]{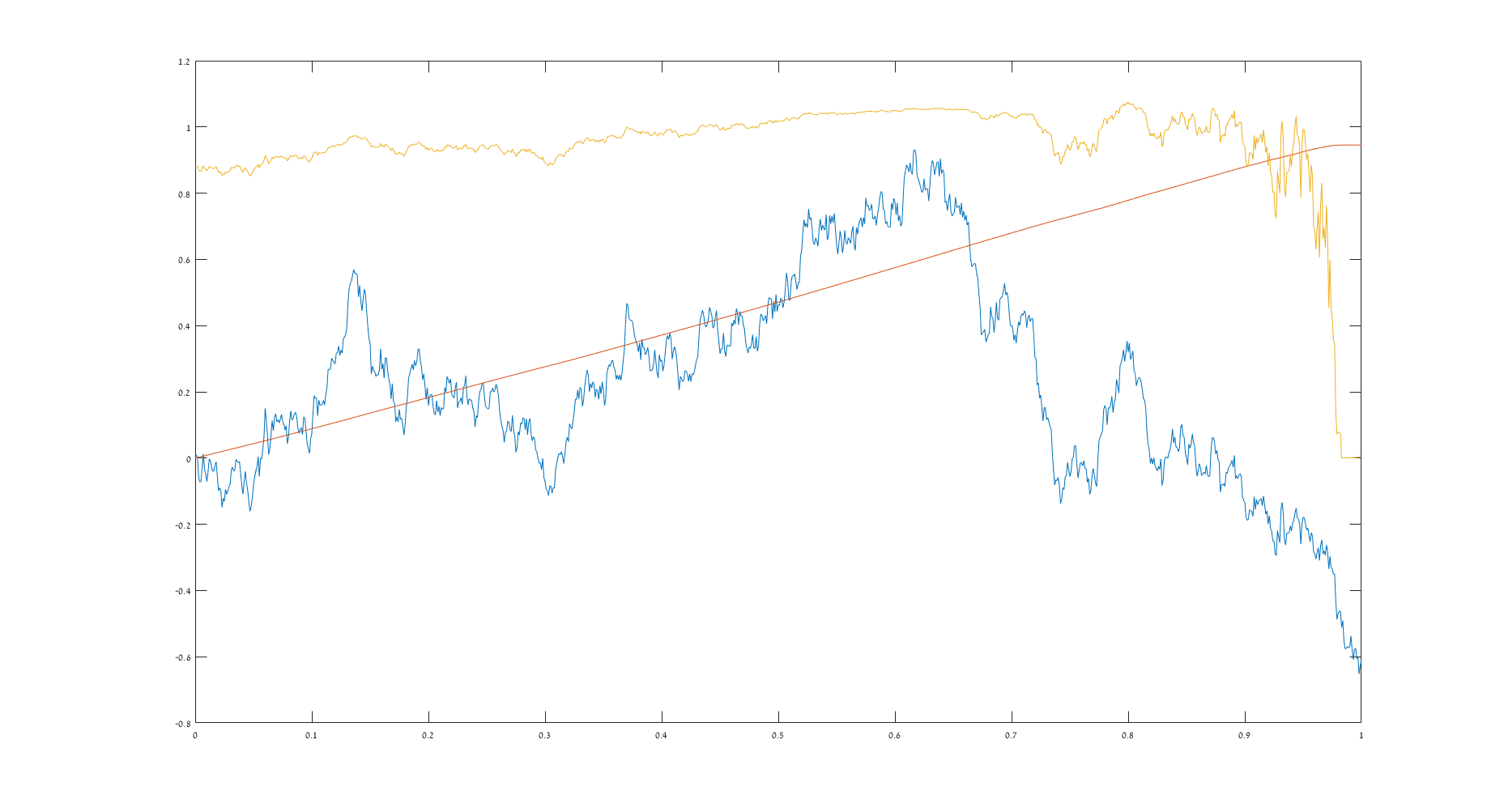}
\caption{This simulation corresponds to the case where $W_1<0$. The Brownian path is in blue, the optimal control
$u\in U(1,0,0)$ is in orange and
$X^{0,u}_t=\int_{0}^t u_s ds$ is in red.}
\end{figure}
\begin{figure}
\centering
\includegraphics[width=0.9\textwidth]{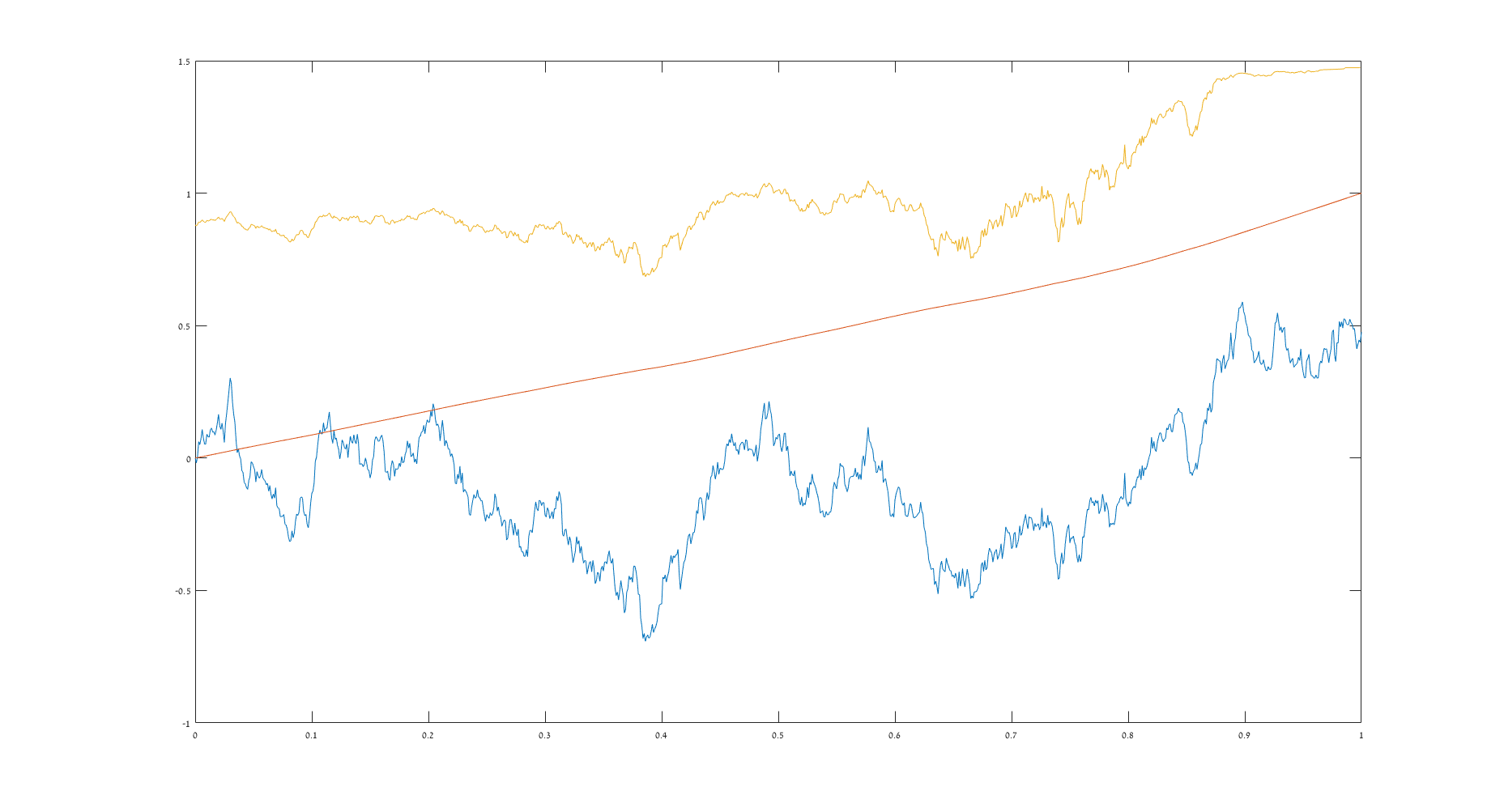}
\caption{This simulation corresponds to the case where $W_1>0$. The Brownian path is in blue, the optimal control
$u\in U(1,0,0)$ is in orange and
$X^{0,u}_t=\int_{0}^t u_s ds$ is in red.}
\end{figure}

\section*{Acknowledgments}
We	
would like to thank the anonymous reviewers for their suggestions and comments which improved the paper.
We also
would like to thank
Peter Bank, Asaf Cohen and Ross Pinsky for valuable discussions.
This research was partially
supported by the ISF grant no 160/17 and the ISF grant no 1707/16.
\bibliographystyle{spbasic}

\end{document}